\documentclass[11pt]{article} % use larger type; default would be 10pt
\usepackage[utf8]{inputenc} % set input encoding (not needed with XeLaTeX)
\usepackage{amsmath,breqn}
\usepackage{amssymb}
\usepackage{graphicx}
\usepackage{pdflscape}
\usepackage{amsthm}
\usepackage{hyperref}

\usepackage{geometry} % to change the page dimensions
\usepackage{graphicx} % support the \includegraphics command and options

\usepackage{booktabs} % for much better looking tables
\usepackage{booktabs} % for much better looking tables
\usepackage{array} % for better arrays (eg matrices) in maths
\usepackage{paralist} % very flexible & customisable lists (eg. enumerate/itemize, etc.)
\usepackage{verbatim} % adds environment for commenting out blocks of text & for better verbatim
\usepackage{subfig} % make it possible to include more than one captioned figure/table in a single float
\usepackage{fancyhdr} % This should be set AFTER setting up the page geometry
\usepackage{tikz}
\usetikzlibrary{positioning}
\usepackage[utf8]{inputenc} % set input encoding (not needed with XeLaTeX)
\usepackage{amsmath}
\usepackage{amsmath,breqn}
\usepackage{amssymb}
\usepackage{graphicx}
\usepackage{epsfig,epstopdf}
\usepackage{pdflscape}
\usepackage{rotating,lscape,afterpage}
\usepackage{amsthm}

\usepackage{geometry} % to change the page dimensions
\usepackage{graphicx} % support the \includegraphics command and options

\usepackage{sectsty}
\allsectionsfont{\sffamily\mdseries\upshape} % (See the fntguide.pdf for font help)
\usepackage[nottoc,notlof,notlot]{tocbibind} % Put the bibliography in the ToC
\usepackage[titles,subfigure]{tocloft} % Alter the style of the Table of Contents

\newtheorem{theorem}{Theorem}[section]
\newtheorem{corollary}{Corollary}[section]
\newtheorem{remk}{Remark}[section]

\newtheorem{lemma}{Lemma}[section]

\numberwithin{equation}{section}

\begin{document}

\title{A note on time-asymptotic bounds with a sharp algebraic rate and a transitional exponent for the sublinear Fujita problem}

\author{D. J. Needham,\\School of Mathematics,\\ University of Birmingham,\\Birmingham, B15 2TT UK,\\d.j.needham@bham.ac.uk\and J. C. Meyer, \\School of Mathematics,\\ University of Birmingham,\\Birmingham, B15 2TT UK,\\j.c.meyer@bham.ac.uk}

\maketitle
\vspace{10pt}

\begin{abstract}
This note establishes sharp time-asymptotic algebraic rate  bounds for the classical evolution problem of Fujita, but with sublinear rather than superlinear exponent. A transitional stability exponent is identified, which has a simple reciprocity relation with the classical Fujita critical blow-up exponent.
\end{abstract}



% Main text
\section{Introduction}
\label{sec:intro}

In this paper we address the classical Fujita problem with sublinear exponent, which takes the form of the following parabolic evolution problem, for $T>0$,
\begin{align}
   & 
   u_t = u_{xx} + [u^p]^+, ~ \forall ~ (x,t)\in D_T, 
   \tag{F1} \\
   &
   u(x,0)=u_0(x)~\forall~x\in \mathbb{R}, 
   \tag{F2} \\
   &
   u~\text{is bounded on} ~ \overline{D}_T. 
   \tag{F3}
\end{align}
Here,
\begin{equation}
    D_T = \{(x,t): x\in \mathbb{R}, t\in (0,T] \} 
\end{equation}
with $x$ being the spatial coordinate and $t$ being time, whilst the nonlinear reaction function $[(\cdot)^p]^+:\mathbb{R}\to\mathbb{R}$ has the simple form,
\begin{equation}
    [u^p]^+ = 
    \begin{cases}
        u^p,~~u\ge0,\\
        0,~~u<0 ,
    \end{cases}
\end{equation}
and we consider the situation when the exponent is sublinear, that is  $0<p<1$. 
In the present context the initial data distribution is restricted so that $u_0\in C(\mathbb{R}) \cap PC^1(\mathbb{R})$, is nontrivial and nonnegative, and has compact support (without loss of generality, we may set sppt$(u_0)\subseteq [-1,1]$); 
for convenience, we henceforth write $u_0\in K^+(\mathbb{R})$ (we remark at the end of the paper how this class of initial data may be considerably extended). 
We refer to this evolution problem as $[F(p)]$, and solutions are regarded as classical, so that $u\in C(\overline{D}_T) \cap C^{2,1}(D_T)$. For superlinear exponents $p>1$ the evolution problem $[F(p)]$ is the classical Fujita problem (see Fujita \cite{Fu}, the reviews of Levine \cite {Lev} and Deng and Levine \cite{Deng}, and the many references therein), and, in the superlinear situation, we recall, in one spatial dimension, that there is a critical blow-up exponent $p=3$, such that when $1<p\le3$, and for any initial data in $K^+(\mathbb{R})$, then $[F(p)]$ has a unique solution, and this solution undergoes spatially local blow-up (in the supnorm) in finite-$t$. 
However, when $p>3$, and the initial data has $||u_0||_{\infty}$ sufficiently small, then $[F(p)]$ has a unique solution, which is global (that is, exits on $\overline{D}_{\infty}$). 
The situation for sublinear exponents $0<p<1$ is significantly different, and this arises due to two features: 
firstly the reaction function is no longer Lipchitz continuous (due to the behaviour as $u \to 0^+$), and so the standard classical theory no longer applies to $[F(p)]$ (however, it is Hölder continuous of degree $p$); 
secondly the curvature of the reaction function on $u>0$ is now negative rather than positive. 
A detailed consideration of $[F(p)]$ with $0<p<1$ has been undertaken in Meyer and Needham \cite{Meyer_Needham_2015} and Aguirre and Escobedo \cite{Ag}. 
It is instructive to summarise the relevant key results established therein in the following:
\begin{theorem}[Aguirre and Escobedo \cite{Ag}, Meyer and Needham \cite{Meyer_Needham_2015}]
    \label{thm1.1}
    Let $0<p<1$ and $u_0\in K^+(\mathbb{R})$. 
    Then for the evolution problem $[F(p)]$ :
    \begin{enumerate}
        \item There exists a global solution $u:\overline{D}_{\infty} \to \mathbb{R}$, and this is unique.
        \item $((1-p)t)^{\frac{1}{(1-p)}}<u(x,t)<(||u_0||_{\infty}^{(1-p)} + (1-p)t)^{\frac{1}{(1-p)}}$ for all $(x,t)\in D_{\infty}$.
        \item For any $T>0$, the limit $u(x,t) \to ((1-p)t)^{\frac{1}{(1-p)}}$ as $|x|\to \infty$ holds  uniformly for $t\in [0,T]$.
        \item The classical parabolic Weak and Strong Comparison Theorems continue to hold.
    \end{enumerate}
\end{theorem}
We observe immediately from the inequality in the second point above that,
\begin{equation}
    u(x,t) \sim ((1-p)t)^{\frac{1}{(1-p)}}~\text{as}~t\to \infty, \label{eqn1.3}
\end{equation}
uniformly for $x\in \mathbb{R}$. 
More specifically, we have,
\begin{equation}
    0<u(x,t) - ((1-p)t)^{\frac{1}{(1-p)}}< \frac{1}{2}||u_0||_{\infty}^{(1-p)}(1-p)^{-1}((1-p)t)^{\frac{p}{(1-p)}}~\text{as}~t\to \infty, \label{eqn1.5}
\end{equation}
uniformly for $x\in \mathbb{R}$. 
Our objective here is to replace the bounds in \eqref{eqn1.5} with sharp estimates in the algebraic rate. 
Our principal result can be stated as:
\begin{theorem}
    \label{thm1.2}
    Let $0<p<1$ and $u:\overline{D}_{\infty}\to \mathbb{R}$ be the solution to $[F(p)]$. 
    Then for each $u_0\in K^+(\mathbb{R})$, the following lower bound holds:
    \begin{equation}
        \label{eqn1.6}
        u(x,t) - ((1-p)t)^{\frac{1}{(1-p)}} \ge c_{-}(x,t,p,u_0) ((1-p)t)^{\frac{(3p-1)}{2(1-p)}} 
    \end{equation}
    as $t\to \infty$ uniformly for $x\in \mathbb{R}$. 
    Conversely, for each $u_0\in K^+(\mathbb{R})$ the following upper bound holds:
    \begin{equation}
        \label{eqn1.7}
        u(x,t) - ((1-p)t)^{\frac{1}{(1-p)}} \le c_{+}(x,t,p,u_0) ((1-p)t)^{\frac{(3p-1)}{2(1-p)}} 
    \end{equation}
    as $t\to \infty$ uniformly for $x\in \mathbb{R}$. 
    Here the positive functions $c_{\pm}(x,t,p,u_0)$ are bounded as $t\to\infty$ uniformly for $x\in \mathbb{R}$, and are explicitly given by the the Gaussian convolution forms,
    \begin{equation}
        c_{-}(x,t,p,u_0) = \frac{(1+||u_0||_{\infty})^{-1}}{2\sqrt{\pi (1-p)}} \int_{-1}^{1}{u_0(s)\exp{\left(-\frac{(s-x)^2}{4t}\right)}}ds
    \end{equation}
    and
    \begin{equation}
        c_{+}(x,t,p,u_0) = \frac{(1-p)^{\frac{(1-2p)}{(1-p)}}}{\sqrt{2\pi (1-p)}} \int_{-\infty}^{\infty}{\mathcal{E}(s;p,u_0)\exp{\left(-\frac{(s-x)^2}{4(t-1)}\right)}}ds  
    \end{equation}
    for each $x\in \mathbb{R}$ and $t>1$, where,
    \begin{equation}
        \label{ee}
        \mathcal{E}(s;p,u_0) =  ((1-p) + \Delta(s,u_0)^{(1-p)})^{\frac{1}{(1-p)}} - (1-p)^{\frac{1}{(1-p)}}, 
     \end{equation}
   and
    \begin{equation}
        \Delta(s,u_0) = \frac{1}{2\sqrt{\pi}} \int_{-1}^{1}{u_0(w)\exp{\left(-\frac{1}{4}(w-s)^2\right)}}dw \label{eqn1.e3}
    \end{equation}
    for all $s\in \mathbb{R}$.
\end{theorem}
A consequence of the inequalities \eqref{eqn1.6} and \eqref{eqn1.7} is:
\begin{corollary}
    \label{cor1.1}
    Let $0<p<1$, and $u:\overline{D}_{\infty}\to \mathbb{R}$ be the solution to $[F(p)]$. 
    Then there is a transitional stability exponent $p=1/3$, such that,
    \begin{itemize}
        \item when $0<p<1/3$ and $u_0\in K^+(\mathbb{R})$
        then $\text{sup}_{x\in\mathbb{R}}(u(x,t) - ((1-p)t)^{\frac{1}{(1-p)}}) \to 0^+$ as $t\to \infty$, and at a precise algebraic rate of $((1-p)t)^{\frac{(3p-1)}{2(1-p)}}$;
        \item when $1/3<p<1$ and $u_0\in K^+(\mathbb{R})$, then $\text{sup}_{x\in\mathbb{R}}(u(x,t) - ((1-p)t)^{\frac{1}{(1-p)}}) \to +\infty$ as $t\to \infty$, and at a precise algebraic rate of $((1-p)t)^{\frac{(3p-1)}{2(1-p)}}$;
        \item when $p=1/3$ and $u_0\in K^+(\mathbb{R})$ then $\text{sup}_{x\in\mathbb{R}}(u(x,t) - ((1-p)t)^{\frac{1}{(1-p)}})$ is bounded below and above as $t\to \infty$ by the positive constants $\overline{c}_-(u_0)$ and $\overline{c}_+(u_0)$ respectively, which are given by,
        \begin{equation}
        \overline{c}_{-}(u_0) = \frac{(1+||u_0||_{\infty})^{-1}}{2\sqrt{2\pi/3}} \int_{-1}^{1}{u_0(s)}ds
    \end{equation}
    and
    \begin{equation}
      \overline{c}_{+}(u_0) = \frac{1}{\sqrt{2\pi}} \int_{-\infty}^{\infty}{\mathcal{E}(s;1/3,u_0)}ds  
    \end{equation}
    with $\mathcal{E}(s;1/3,u_0)$ as given via \eqref{ee}.
    \end{itemize}
\end{corollary}
This corollary is a direct consequence of Theorem 1.2 (and needs no further proof) and the results can be interpreted in terms of the spatio-temporal stability of the spatially homogeneous state $u=u_h(t)\equiv ((1-p)t)^{\frac{1}{(1-p)}}$:
\begin{remk}
    The spatially homogeneous state $u=u_h(t)$, when subject to initial disturbances in $K^+(\mathbb{R})$, is asymptotically stable when $0<p<1/3$, is Liapunov stable when $p=1/3$, and is unstable when $1/3<p<1$.
\end{remk}
We also have:
\begin{remk}
    The problem $[F(p)]$ can be considered on the higher dimensional spatial domain $\mathbb{R}^N$ for $N\in\mathbb{N} = 2,3,\ldots$. 
   Theorem \ref{thm1.1} continues to hold without change. 
   It is also straightforward to adapt Theorem \ref{thm1.2}. 
   The key change is that the algebraic power of $t$ on the left hand side of both inequalities now becomes $((N+2)p-N)/(2(1-p))$, whilst the integrals in the remaining terms have their natural modification into N-dimensional multiple integrals.
   The conclusions of Corollary \ref{cor1.1} continue to hold except now the transitional stability exponent becomes $p=p_c^-(N)$ which is given by,
   \begin{equation}
       p_c^-(N) = N(N+2)^{-1}.
   \end{equation}
   \end{remk}
An interesting observation now is that when $p>1$, in higher spatial dimensions, the Fujita critical blow-up exponent becomes $p=p_c^+(N)$ where,
\begin{equation}
   p_c^+(N) = 1 + 2/N.
\end{equation}
Thus we have the interesting reciprocal relationship,
\begin{equation}
   p_c^+(N) p_c^-(N) = 1. \label{4}
\end{equation}
We observe that, for small initial data, the critical exponent $p_c^+(N)$ is brought about by the balancing of two processes \emph{relative to the background trivial equilibrium state}: the weak decay due to linear diffusion balancing the weakly nonlinear growth due to the degenerate reaction term. 
However it is a different balance which determines the transitional exponent $p_c^-(N)$ with \emph{the background now being the strongly nonlinear nontrivial homogeneous state $u_h(t)$} and the balance now being the weak decay due to diffusion with the linearised reaction, both \emph{now relative to} $u_h(t)$. 
The precise nature of these respective mechanisms results in the exact reciprocity in relation 
\eqref{4}.

The remainder of the paper concerns the proof of Theorem \ref{thm1.2}.

\section{Preliminary constructions}
\label{sec:prelims}

In this section we introduce and examine two functions which will play a subsequent role in constructing both a suitable subsolution and supersolution to $[F(p)]$. 
The first is a familiar function which uniquely solves the evolution problem [IVP$\mathcal{D}$] for the linear diffusion equation, namely,
\begin{align}
    & 
    \mathcal{D}_t = \mathcal{D}_{xx},~\forall ~(x,t)\in D_{\infty}, 
    \tag{D1} \\
    &
    \mathcal{D}(x,0;u_0) = u_0(x),~ \forall ~ x\in \mathbb{R}, 
    \tag{D2} \\
    &
    \mathcal{D}~\text{is bounded on}~\overline{D}_T~\text{for each}~T>0, \tag{D3}
\end{align}
and is given by,
\begin{equation}
   \mathcal{D}(x,t;u_0) = \frac{1}{2\sqrt{\pi t}} \int_{-1}^{1}{u_0(s)\exp{\left(-\frac{(s-x)^2}{4t}\right)}}ds
\end{equation}
for all $(x,t)\in D_{\infty}$. We recall that
\begin{equation}
   \mathcal{D}\in C(\overline{D}_{\infty})\cap C^{2,1}(D_{\infty}), \label{eqn2.1}
\end{equation}
and satisfies the inequalities,
\begin{equation}
   0<\mathcal{D}(x,t;u_0) < \frac{1}{2\sqrt{\pi t}} \int_{-1}^{1}{u_0(s)}ds \label{eqn2.2}   
\end{equation}
for all $(x,t)\in D_{\infty}$, together with,
\begin{equation}
    \mathcal{D}(x,t;u_0) < \frac{1}{2\sqrt{\pi t}} \left(\int_{-1}^{1}{u_0(s)}ds\right) \exp{\left(-\frac{(|x|-1)^2}{4t}\right)}, \label{eqn2.3}
\end{equation}
for all $(x,t)\in D_\infty$ such that $|x|\geq 1$.

Next consider the linear evolution problem [IVP$\mathcal{W}$], namely,
\begin{align}
    & \mathcal{W}_t = \mathcal{W}_{xx} + p((1-p)t)^{-1} \mathcal{W}~\forall ~ (x,t)\in \mathbb{R} \times (1,\infty),\label{eqnw} \tag{W1} 
    \\
    & \mathcal{W}(x,1;p,u_0) = \mathcal{E}(x;p,u_0)~\forall ~ x\in \mathbb{R}
    \tag{W2}
    \\
    & \mathcal{W}~\text{is bounded on}~\mathbb{R} \times [1,T]~\text{for each}~T>1 \tag{W3}
\end{align}
with $\mathcal{E}$ given by \eqref{ee}. 
We observe that,
\begin{equation}
    0< \mathcal{E}(x;p,u_0) \le  ((1-p) + ||u_0||_{\infty}^{(1-p)})^{\frac{1}{(1-p)}} - (1-p)^{\frac{1}{(1-p)}}
\end{equation}
for all $x\in \mathbb{R}$. 
We note that [IVP$\mathcal{W}$] has the unique and global solution $\mathcal{W} \in C^{2,1}(\mathbb{R} \times [1,\infty))$ given by,
\begin{equation}
    \mathcal{W}(x,t;p,u_0) = \frac{t^{\frac{p}{(1-p)}}}{2\sqrt{\pi(t-1)}} \int_{-\infty}^{\infty}{\mathcal{E}(s;p,u_0)\exp{\left(-\frac{(s-x)^2}{4(t-1)}\right)}}ds  \label{eqn2.4}
\end{equation}
for $(x,t)\in \mathbb{R} \times (1,\infty)$.
\begin{remk}
    The significance of the linear parabolic PDE in (\ref{eqnw}) arises from it being the formal linearisation of the PDE in $[F(p)]$ about the homogeneous state $u=u_h(t)$.
\end{remk}
We again readily establish that,
\begin{equation}
    \mathcal{W}\in C(\mathbb{R} \times [1,\infty)) \cap C^{2,1}(\mathbb{R} \times (1,\infty)), \label{eqn2.5}
\end{equation}
whilst we have the bound,
\begin{equation}
    ||\mathcal{W}(\cdot,t)||_{\infty} \le \text{min}\left(t^{\frac{p}{(1-p)}} ||\mathcal{E}(\cdot;p,u_0)||_{\infty}, \frac{t^{\frac{p}{(1-p)}}}{2\sqrt{\pi(t-1)}}I(p,u_0)\right)  \label{eqn2.6}
\end{equation}
for $t\in (1,\infty)$, on using \eqref{eqn2.1}. 
Here
\begin{equation}
    I(p,u_0) = \int_{-\infty}^{\infty}{\mathcal{E}(s;p,u_0)}ds. 
\end{equation}
We now use the above functions in the following constructions.

\section {The key subsolution and supersolution to [F(p)]}
\label{sec:subsupsol}

Throughout this section, for any $T>T_0\ge0$ and function $\psi \in C(\mathbb{R} \times  [T_0,T)) \cap C^{2,1}(\mathbb{R} \times (T_0,T))$, we introduce the mapping $\mathcal{N}:C^{2,1}(\mathbb{R}\times (T_0,T)) \to C(\mathbb{R}\times (T_0,T))$ as
\begin{equation}
    \mathcal{N}(\psi) \equiv \psi_t - \psi_{xx} - [\psi^p]^+. \label{eqn3.1}
\end{equation} 
We next introduce the function $\overline{u}^+\in C(\overline{D}_{\infty}) \cap C^{2,1}(D_{\infty})$ such that,
\begin{equation}
    \overline{u}^+(x,t) = \left((1-p)t + \mathcal{D}(x,t;u_0)^{(1-p)} \right)^{\frac{1}{(1-p)}}  
\end{equation}
for all $(x,t)\in \overline{D}_{\infty}$. 
We can now appeal directly to \cite[Chapter 9, Proposition 9.2]{Meyer_Needham_2015} to establish that, for any $T>0$, then $\overline{u}^+$ is a supersolution to $[F(p)]$ on $\overline{D}_T$; we then have:
\begin{lemma}
    \label{lem3.1}
    Let $0<p<1$ and $u:\overline{D}_{\infty} \to \mathbb{R}$ be the solution to $[F(p)]$. 
    Then, for any $T>0$,
    \begin{equation}
        u(x,t)\le \left((1-p)t + \mathcal{D}(x,t;u_0)^{(1-p)} \right)^{\frac{1}{(1-p)}}
    \end{equation}
    for all $(x,t)\in \overline{D}_T$.
\end{lemma}

\begin{proof}
    Recalling that $\overline{u}^+$ is a supersolution to $[F(p)]$ on $\overline{D}_T$, then an application of the Weak Comparison Theorem (which is validated via Theorem \ref{thm1.1}(4)) leads directly to the result.
\end{proof}
It follows directly from this inequality and Theorem \ref{thm1.1}(2) that,
\begin{equation}
    ((1-p))^{\frac{1}{(1-p)}} < u(x,1) \le \left((1-p) + \Delta(x,u_0)^{(1-p)} \right)^{\frac{1}{(1-p)}} \label{eqn3.4}
\end{equation}
for all $x\in \mathbb{R}$, and this will be the starting point of the second construction that is developed below.

We now introduce our key subsolution. We have:
\begin{lemma}
    \label{lem3.2}
    For each $T>0$ the function $\underline{u}: \overline{D}_{\infty} \to \mathbb{R}$, given by
    \begin{equation}
        \underline{u}(x,t) = \left((1-p)t + (1+||u_0||_{\infty})^{-1}\mathcal{D}(x,t;u_0) \right)^{\frac{1}{(1-p)}} \label{eqnusub}
    \end{equation}
    for all $(x,t)\in \overline{D}_{\infty}$, is a subsolution to $[F(p)]$ on $\overline{D}_T$.
\end{lemma}

\begin{proof}
    Fix $T>0$ and observe that $\underline{u} \in C(\overline{D}_{\infty})\cap C^{2,1}(D_{\infty})$. 
    Next, using the inequalities \eqref{eqn2.2}, it is readily confirmed that $\underline{u}$ is bounded on $\overline{D}_T$. 
    Secondly, since $0<p<1$ we have
    \begin{align}
        \underline{u}(x,0) & = \left( (1+||u_0||_{\infty})^{-1}\mathcal{D}(x,0;u_0) \right)^{\frac{1}{(1-p)}} \nonumber
        \\
        & = \left( (1+||u_0||_{\infty})^{-1} u_0(x) \right)^{\frac{1}{(1-p)}} 
        \nonumber
        \\
        & \le (1 + ||u_0||_{\infty})^{-1} u_0(x) \nonumber
        \\
        & \le u_0(x) \nonumber
    \end{align}
    for all $x\in \mathbb{R}$. 
    Finally, for $(x,t)\in D _T$,
    \begin{align*}
        \mathcal{N}(\underline{u})(x,t) & = (\underline{u}_t - \underline{u}_{xx} - [\underline{u}^p]^+)(x,t) 
        \\
        & = -\frac{p}{(1-p)^2} (1 + ||u_0||_{\infty})^{-2} \mathcal{D}_{x}(x,t;u_0)^2 \left( (1-p)t + (1 + ||u_0||_{\infty})^{-1} \mathcal{D}(x,t;u_0)\right)^{\frac{(2p-1)}{(1-p)}} 
        \\ 
        &\le 0 
    \end{align*}
    which completes the proof.
\end{proof}

Next we have the key supersolution:
\begin{lemma}
    \label{lem3.3}
    For each $T>1$ the function $\overline{u}: \mathbb{R} \times [1,\infty) \to \mathbb{R}$, given by
    \begin{equation}
        \overline{u}(x,t) = ((1-p)t)^{\frac{1}{(1-p)}} + \mathcal{W}(x,t;p,u_0) \label{eqnusup}
    \end{equation}
    for all $(x,t)\in \mathbb{R} \times [1,\infty)$, is a supersolution to $[F(p)]$ on $\mathbb{R}\times [1,T]$. 
\end{lemma}

\begin{proof}
    Fix $T>1$ and observe that $\overline{u}\in C(\mathbb{R} \times [1,\infty))\cap C^{2,1} (\mathbb{R} \times (1,\infty))$. 
    It immediately follows from \eqref{eqn2.6} that $\overline{u}$ is bounded on $\mathbb{R}\times [1,T]$. 
    Next we have, 
    \begin{align*}
        \overline{u}(x,1) & = ((1-p))^{\frac{1}{(1-p)}} + \mathcal{W}(x,1;p,u_0) 
        \\
        & = (1-p)^{\frac{1}{(1-p)}} + \mathcal{E}(x;p,u_0) 
        \\ 
        & = (1-p)^{\frac{1}{(1-p)}} + \left(((1-p) + \Delta(x,u_0)^{(1-p)})^{\frac{1}{(1-p)}} - (1-p)^{\frac{1}{(1-p)}}\right) 
        \\
        & = ((1-p) + \Delta(x,u_0)^{(1-p)})^{\frac{1}{(1-p)}} 
        \\
        & \ge u(x,1) 
    \end{align*}
    for all $x\in \mathbb{R}$, via \eqref{eqn3.4}. 
    Now, for $(x,t) \in \mathbb{R} \times (1,T]$, via \eqref{eqnw} we have
    \begin{align*}
        \mathcal{N}(\overline{u})(x,t) & = (\overline{u}_t - \overline{u}_{xx} - [\overline{u}^p]^+)(x,t) 
        \\
        & = ((1-p)t)^{\frac{p}{(1-p)}} + (\mathcal{W}_t - \mathcal{W}_{xx})(x,t) - [(((1-p)t)^{\frac{1}{(1-p)}} + \mathcal{W}(x,t;p,u_0))^p]^+ 
        \\ 
        & = ((1-p)t)^{\frac{p}{(1-p)}} + p((1-p)t)^{-1} \mathcal{W}(x,t) -  [(((1-p)t)^{\frac{1}{(1-p)}} + \mathcal{W}(x,t;p,u_0))^p]^+ 
        \\
        & = ((1-p)t)^{\frac{p}{(1-p)}} + p((1-p)t)^{-1} \mathcal{W}(x,t) - (((1-p)t)^{\frac{p}{(1-p)}} 
        \\
        & \quad + p(((1-p)t)^{\frac{1}{(1-p)}}+ \theta (x,t)\mathcal{W}(x,t))^{-(1-p)}\mathcal{W}(x,t)) 
        \\
        & \ge 0
    \end{align*}
    since $0<p<1$, where we note that $\theta(x,t) \in (0,1)$ exists via the Mean Value Theorem. 
    The proof is complete.  
\end{proof}
We now have:
\begin{corollary}
    \label{cor3.1}
    Let $0<p<1$, and $u:\overline{D}_{\infty}\to \mathbb{R}$ be the solution to $[F(p)]$. 
    Then for each $(x,t)\in \mathbb{R}\times [1,\infty)$,
    \begin{equation}
        \underline{u}(x,t) \le u(x,t) \le \overline{u}(x,t). \label{eqnc}
    \end{equation}
\end{corollary}

\begin{proof}
    This follows from Lemma \ref{lem3.2} and Lemma \ref{lem3.3} on use of the Weak Comparison Theorem (via Theorem \ref{thm1.1}(4)).
\end{proof}
It is now straightforward to establish the inequalities \eqref{eqn1.6} and \eqref{eqn1.7} in Theorem \ref{thm1.2} directly from \eqref{eqnusub} and \eqref{eqnusup} (together with the bounds on $\mathcal{D}$ and $\mathcal{W}$ obtained in section \ref{sec:prelims}), which completes the proof of this theorem. 
To finish the paper we make the final observation:
\begin{remk}
    The containing set $K^+(\mathbb{R})$ for initial data in the definition of the evolution problem $[F(p)]$ can be considerably broadened to allow for all nontrivial, non-negative functions in $C(\mathbb{R}) \cap L^1(\mathbb{R})$ which have zero limit as $|x|\to \infty$. 
    This extension follows the above very closely, and requiring only very minor technical modifications. 
    Similarly, modifications to generalise to higher spatial dimensions follow the obvious adaptations.
\end{remk}

\section*{Funding Statement}
This research did not receive any specific grant from funding agencies in the public, commercial, or not-for-profit sectors.

\section*{Declaration of Interests}
The authors have nothing to declare.

\bibliographystyle{plain}

\end{document}